\documentclass[11pt]{amsart}
\usepackage{bbm}
\usepackage{verbatim}
\usepackage{graphicx,color}
\usepackage{enumerate}
\usepackage[all]{xy}
\usepackage[hyperindex=true,plainpages=false,colorlinks=false,pdfpagelabels]{hyperref}

\theoremstyle{plain}
\newtheorem{The}{Theorem}
\newtheorem*{The*}{Theorem}
\newtheorem{Pro}{Proposition}
\newtheorem{Lem}{Lemma}
\newtheorem{Cor}{Corollary}
\newtheorem*{Cor*}{Corollary}

\theoremstyle{definition}

\newtheorem{Rem}{Remark}

\newtheorem*{Rem*}{Remark}

\numberwithin{equation}{section}

\DeclareMathOperator{\SL}{SL}
\DeclareMathOperator{\SLt}{SL(2,\mathbb C)}

\DeclareMathOperator{\SU}{SU}

\DeclareMathOperator{\Id}{Id}

\DeclareMathOperator{\vol}{vol}

\DeclareMathOperator{\res}{res}
\DeclareMathOperator{\Eig}{Eig}

\newcommand{\dvector}[1]{{\left(\begin{matrix}#1\end{matrix}\right)}}
\DeclareMathOperator{\dbar}{\bar\partial}

\DeclareMathOperator{\pdeg}{par-deg}

\newcommand{\R}{\mathbb{R}}

\newcommand{\C}{\mathbb{C}}

\newcommand{\Z}{\mathbb{Z}}

\newcommand{\CP}{\mathbb{CP}}

\begin{document}

\title[Abelianization of Fuchsian Systems]{Abelianization of Fuchsian Systems on a $4-$punctured sphere and applications}

\author{Lynn Heller }

\address{ Institut f\"ur Mathematik\\  Universit{\"a}t T\"ubingen\\ Auf der Morgenstelle
10\\ 72076 T\"ubingen\\ Germany
}
  \email{lynn-jing.heller@uni-tuebingen.de}

\author{Sebastian Heller}
\address{ Institut f\"ur Mathematik\\  Universit{\"a}t T\"ubingen\\ Auf der Morgenstelle
10\\ 72076 T\"ubingen\\ Germany
}
\email{heller@mathematik.uni-tuebingen.de}


\thanks{Both authors were supported by the DFG}

\begin{abstract} 

In this paper we consider special linear Fuchsian systems of rank $2$ on a $4-$punctured sphere and the corresponding parabolic structures. Through an explicit abelianization procedure we construct a $2-$to$-1$ correspondence between 
flat line bundle connections on a torus and these Fuchsian systems. This naturally equips the moduli space of flat 
$\SLt$ connections on a $4-$punctured sphere with a new set of Darboux coordinates. Furthermore, we apply our 
theory to give a complex analytic proof of Witten's formula for the symplectic volume of the moduli space of unitary flat 
connections on the $4-$punctured sphere.

 \end{abstract}
\maketitle


\section{Introduction}
Moduli spaces $\mathcal M$ of flat $G-$connections on a compact Riemann surface $\Sigma$ are 
equipped with interesting geometric structures. Prominent examples beyond the (abelian) line bundle case
 are provided by the special unitary
group and the special linear group: For  $G = \SU(n)$ the moduli space $\mathcal M$ inherits
 a natural  K\"ahler metric and  for $G=\SL(n,\C)$ the moduli space is even
hyper-K\"ahler (see for example \cite{At,Hi1}). This correspondence can be generalized to the case of punctured Riemann surfaces by prescribing the conjugacy classes of the local monodromies, i.e., for suitable boundary conditions on
the objects of interest \cite{Do}.

In the general case where neither $G$ nor the fundamental group $\pi_1(\Sigma)$ are abelian, it is hard to find a unified and explicit description
of the moduli space $\mathcal M$ with all its geometric structures. 
For example, it is not known how to explicitly represent unitary connections on a Riemann surface in a way which makes
its K\"ahler structure visible. Further, it is not possible to 
 to see 
all K\"ahler structures at once in a computable way for $G=\SL(2,\C)$. 
The main reason for this lack of understanding is due to the fact that it is generally not possible to compute the monodromy 
representation of an irreducible connection. Recent progress towards the understanding of the hyper-K\"ahler geometry
of the moduli space of flat $\SLt$ connections was made in \cite{GMN} by an abelianization procedure based on the WKB analysis along
so-called spectral networks. However, this work does not take the underlying holomorphic structures (respectively 
parabolic structures in the presence of punctures) into full account. But this seems to be necessary
for a complete understanding of these moduli spaces and for some applications such as the integrable systems approach to 
harmonic maps (see \cite{He3}).

In this paper we carry out an abelianization procedure for flat $\SLt $ connections on a $4-$punctured sphere
which does not only make the underlying parabolic structures as transparent as possible but also sheds new light on the
K\"ahler structure of the moduli space of flat $\SU(2)$ connections on the $4-$punctured sphere.
The starting point of our theory is the following well-known fact which is a special instance of 
the Riemann Hilbert correspondence (see for example \cite{Dek} for the general treatment of the rank $2$ case on a $n-$punctured sphere): All representations
$\pi_1(\CP^1\setminus\{z_0,..,z_3\},*)\to \SLt$
can be realized as the monodromy representation of a Fuchsian system, i.e., of a meromorphic connection $\nabla$ on the trivial rank two bundle 
$\C^2\to\CP^1$ with first order poles at the singular points $z_0,..,z_3\in\CP^1.$ We are interested in the case where the monodromy representation is unitary up to conjugation.
Since the conjugacy classes of  the monodromies around a puncture (local monodromies) 
are generally determined by the residues of the connections, we restrict ourselves to the case
of trace-free residues with real eigenvalues $\pm\rho_i$ such that 
$\rho_i\in]0,\frac{1}{2}[$ (excluding singular cases). 
The eigenlines with respect to the positive eigenvalues $\rho_i$ of the residues of a Fuchsian system determine flags of $\C^2$ at
the singular points 
 together with weight filtrations induced by the eigenvalues. 
 This gives rise to a parabolic structure associated to a Fuchsian system.
 The notion of stability of parabolic structures can be defined and it turns out that this notion is naturally connected
to the question of unitarizable monodromy (\cite{MS}): For every stable parabolic structure there exists a unique
 compatible Fuchsian system whose (irreducible) monodromy representation is unitary up to conjugation.
In section \ref{sec:fuchs} we give more details on the relationship between Fuchsian systems and 
parabolic structures. In particular, we recall a useful parametrization of Fuchsian systems from \cite{LS} and discuss
stability issues of the corresponding parabolic structures.

In section \ref{sec:abel} we shift our attention to the moduli spaces of parabolic structures and
flat connections and study them
 via abelianization.
The space of special linear Fuchsian systems on a $4-$punctured sphere with prescribed residue eigenvalues $\pm\rho_i$ 
 is a complex two dimensional variety, while the moduli space of (semi-)stable parabolic structures is
a projective line equipped with its natural complex structure \cite{ArLy,LS}. The forgetful map from Fuchsian systems to parabolic structures gives rise to an affine line bundle whose underlying vector bundle consists
of parabolic Higgs fields, i.e., meromorphic $\mathfrak{sl}(2,\C)-$valued $1-$forms with first order poles fixing a given parabolic structure when added to a compatible Fuchsian system.
Generically, the eigenlines of parabolic Higgs fields are only well-defined on a torus given by the double cover of the Riemann sphere branched over the singular points.
The eigenlines determine the parabolic structure and vice versa. This gives rise to a $2-$to$-1$ correspondence
between the Jacobian of the torus and the moduli space of parabolic structures.
This correspondence extends to flat line bundle connections on the one side
and flat $\SLt$ connections on the other 
in the following way (Theorem \ref{two-to-one}):
The eigenlines of the parabolic Higgs field span the rank $2$ bundle away from the branch divisor and the connection gives rise to
meromorphic line bundle connections on the eigenlines with first order poles (and fixed residues $\tfrac{1}{2}$) at the branch 
divisor. Factorizing out the poles, i.e., tensoring with a special flat meromorphic line bundle of degree $2$, yields 
(ordinary) flat line bundles on the torus. Moreover, the second fundamental forms of the flat $\SLt$ connections with 
respect to the line subbundles are uniquely determined by the underlying holomorphic structure of the line bundle
(Proposition \ref{explicit_coeff}).
By choosing Darboux coordinates on the moduli space of flat line bundles over the torus
 we also obtain a new set of Darboux coordinates for the natural holomorphic symplectic structure on the moduli space of 
 flat $\SLt$ connections on
the $4-$punctured sphere with prescribed local monodromies, see Theorem \ref{Darboux_coord}.

In the last section, section \ref{symp_volume}, we apply the results and methods from section \ref{sec:abel} to compute the symplectic 
volume of the moduli space $\mathcal M$ of special unitary connections
on the 4-punctured sphere with prescribed local monodromies. This is a special case of Witten's formula 
\cite{W}. 
We give an alternative
complex analytic proof of this formula: Applying Theorem \ref{Darboux_coord} we can write down an explicit
representative of the cohomology class of the symplectic form on the Jacobian which double covers $\mathcal M$. This $2-$form can be easily integrated over the Jacobian and yields the symplectic volume of 
$\mathcal M.$

\section{Fuchsian Systems}\label{sec:fuchs}
Let $M = \CP^1 \setminus \{z_0,..,z_3\}$ be a $4-$punctured Riemann sphere. 
By applying a Moebius transformation we can always assume that $z_0=[1:0],$ $z_1=[1:1],$ $z_2=[0:1],$  and
$z_3=[m:1]$ for a suitable $m\in\C\setminus\{0,1\}.$
We consider Fuchsian systems on $M$ which are systems of differential equations describing parallel sections of the trivial rank $2$ vector bundle $V= \underline \C^2$ over $M$ given by a meromorphic connection of the form
\begin{equation}\label{Fuchs2}\nabla=d+A_1\frac{dz}{z-1}+A_2\frac{dz}{z}+A_3\frac{dz}{z-m}.\end{equation}
Note that $\nabla$ also has a first order pole at $z=\infty$ with residue $-A_1-A_1-A_3=A_0.$

The Riemann-Hilbert Problem is solved for the $\SLt$ case and gives a correspondence between $\SLt$ representations of the first fundamental group $\pi_1(M,*)$ and trace-free Fuchsian systems.
Unitarizable representations are those representations lying in the $\SLt$ conjugacy classes of $\SU(2)$ representations. A natural question is which Fuchsian systems correspond to unitarizable representations. 
There are necessary conditions (the Biswas conditions \cite{Bis3}) on the eigenvalues of the $A_i$ for a Fuchsian system to have unitarizable monodromy, but these conditions are far from being sufficient.
Nevertheless, it is natural to study 
Fuchsian systems on the 4-punctured sphere $\CP^1\setminus\{z_0,..,z_3\}$ 
with prescribed conjugacy classes of the local monodromies.
In view of the Biswas conditions we assume that the eigenvalues $\pm\rho_i$ of $A_i$ are real and lie the interval $]-\frac{1}{2},\frac{1}{2}[.$
In order to exclude the degenerated cases, we restrict to the case that 
\[-\rho_i<0<\rho_i\]
for $i=0,..,3.$ Clearly, the local monodromies around the singularity $z_i$ lies in the conjugacy class of
\[\dvector{\exp(2\pi i\rho_i) &0\\0 &\exp(-2\pi i\rho_i)},\]
and the choice of the conjugacy class of the local monodromies  is equivalent to the choice of the eigenvalues of the residues $A_i.$

\subsection{Parabolic structures}
A Fuchsian system as in \eqref{Fuchs2} gives rise to a parabolic structure as follows  (for more details see
\cite{MS, Bis3, ArLy} or \cite{Pir}):
The underlying holomorphic vector bundle $V$ of a Fuchsian system is the trivial holomorphic bundle $\C^2\to\CP^1.$ The residue
$A_i$ of the connection $\nabla$ at the singularity $z_i$ give rise to a complex line
\[E_i=\ker(A_i-\rho_i\Id)\]
(where $\rho_i>0$ is the positive eigenvalue)
together with a filtration
\[0\subset E_i\subset V_{z_i}\]
of the fiber of $V$ at $z_i.$ Then the parabolic structure is given by
these filtrations over the singularities together with the corresponding weight filtration $(\rho_i,-\rho_i),$
i.e., the line $E_i$ is equipped with the weight $\rho_i$ while $V_{p_i}\setminus E_i$ is equipped with the weight
$-\rho_i.$
Note that the parabolic degree of $V$
\[\pdeg V = \deg V+\sum_i\sum\text{eigenvalues of } A_i=\sum_i(\rho_i-\rho_i) = 0\]
automatically vanishes in our situation. A holomorphic line subbundle $L\subset V$ is equipped
with the induced parabolic degree
\[\pdeg L=\deg L+\sum_i\gamma_i,\]
where (for $i=0,..,3$) $\gamma_i$ is defined to be $\rho_i$ if $L_{p_i}=E_i$ and $-\rho_i$ otherwise.
The parabolic structure is called stable (respectively semi-stable) if the parabolic degree is negative (respectively non-positive) for all holomorphic line subbundles $L$: $\pdeg L<0,\, (\leq 0).$ By \cite{MS,Biq} and because of the Riemann Hilbert correspondence, every stable
parabolic structure admits a Fuchsian system with unitarizable monodromy representation.
Moreover, up to isomorphisms respectively conjugations, this correspondence between stable parabolic structures
and irreducible unitary monodromy representations on a punctured sphere is $1-$to$-1$. Additionally, reducible unitary
monodromy representations give rise to strictly semi-stable parabolic structures.

In this paper we are interested in the moduli space of Fuchsian systems on the 4-punctured sphere
with prescribed conjugacy classes of the local monodromies. Parabolic stability is an open condition.
Hence, 
 a generic
Fuchsian system with prescribed eigenvalues of the residues
induces a stable parabolic structure if there exists one Fuchsian system with these eigenvalues whose parabolic structure is stable.
In our case a criterion for the stability follows from
\cite{Bis3}: For given $\rho_i,$
there exists a Fuchsian system inducing a stable parabolic structure if and only if
\begin{equation}\label{Biswas_criterion}
1+\rho_{\sigma(3)}>\rho_{\sigma(0)}+\rho_{\sigma(1)}+\rho_{\sigma(2)}>\rho_{\sigma(3)}
\end{equation}
for all permutations $\sigma\in\mathfrak S(\{0,1,2,3\}).$ We will give a short proof of this (in the $4-$puncture case) in section \ref{stability}.

\subsection{Parabolic Higgs fields}
Consider a Fuchsian system $\nabla$ and its induced parabolic structure as above. If we add to $\nabla$ 
a  meromorphic 1-form \[\Psi\in H^{1,0}(\CP^1\setminus\{z_0,..,z_3\},\mathfrak{sl}(2,\C))\]
with first order poles, the induced parabolic structure will change in general.
The condition that $\nabla+\Psi$ has the same parabolic structure as
$\nabla$ is that the eigenlines
$E_i$ of the positive eigenvalues $\rho_i>0$  are in the kernel of the residues of $\Psi$
at the singularities $z_i.$ If this condition is satisfied, $\Psi$ is called a {\em parabolic Higgs field}.

Then we observe:
\begin{Pro}
For a generic special linear Fuchsian system on the 4-punctured sphere, the space of
parabolic Higgs fields is complex 1-dimensional. In general, the determinant of a parabolic Higgs field
is a non-zero meromorphic quadratic differential with first order poles on $\CP^1$, i.e., a constant multiple of
$\frac{(dz)^2}{\Pi_{i=1}^3 (z-z_i)}.$
\end{Pro}

\subsection{Concrete formulas}\label{explicit1}
Throughout this paper, we make use of the 
following explicit parametrization of trace-free Fuchsian systems on a $4-$punctured sphere \cite{LS}.
Let $\rho_i>0$ and let $\rho=\rho_0-\rho_1-\rho_2-\rho_3.$
By introducing a complex parameter $u$ (representing the parabolic structure) we can set
\begin{equation}\label{concrete_parabolic_structure}
\begin{split}
A^u_1&=\dvector{-\rho_1-\rho &2 \rho_1+\rho \\ -\rho & \rho_1+\rho},\, \,\,\,A^u_2=\dvector{-\rho_2  &0 \\ \rho & \rho_2},\\
A^u_3&=\dvector{-\rho_3 &2\rho_3 u \\ 0 & \rho_3} ,\, \,\,\,
A^u_0=-A^u_1-A^u_2-A^u_3=\dvector{\rho_0 &-\rho_0-\rho_1+\rho_2+\rho_3-2\rho_3 u \\ 0 & -\rho_0}.
\end{split}
\end{equation}

Then the connection 
\[\nabla^u:=d+A^u_1\frac{dz}{z-1}+A^u_2\frac{dz}{z}+A^u_3\frac{dz}{z-m}\] is a Fuchsian system with poles at $z_0=\infty,$ $z_1=1,$ $z_2=0$ and $z_3=m$ whose local monodromies are determined by $\pm\rho_0$  $,\pm\rho_1$, $\pm\rho_2$  and $\pm\rho_3,$ respectively. Moreover, for
\begin{equation}\label{concrete_Higgs}
\begin{split}
\Psi_1=\dvector{u&-u \\ u & -u},\, \,\,\,\, \Psi_2&=\dvector{0 &0 \\ 1-u & 0},\\
\Psi_3=\dvector{-u &u^2 \\ -1 & u}, \, \,\,\,\Psi_0&=-\Psi_1-\Psi_2-\Psi_3=\dvector{0 &u-u^2 \\ 0 & 0} 
\end{split}
\end{equation}
the $1-$form
\begin{equation}\label{psi^u}
\Psi^u=\Psi:=\Psi_1\frac{dz}{z-1}+\Psi_2\frac{dz}{z}+\Psi_3\frac{dz}{z-m}
\end{equation}
 is a Higgs field with respect to the induced parabolic structure of 
$\nabla^u.$ 
Thus, a generic monodromy representation of the fundamental group of the $4-$punctured sphere with given local monodromies can be realized 
by a
unique 
\[\nabla^{u,\lambda}:=\nabla^u+\lambda\Psi, \quad\lambda\in \C\] up to conjugation. The eigenlines of the positive eigenvalues $\rho_i>0$  of the residues are
\begin{equation}\label{concrete_u}
\begin{split}
\Eig(\res_{z_0}\nabla^{u,\lambda},\rho_0) &=\C\dvector{1\\0} \,\,\,\, \Eig(\res_{z_1}\nabla^{u,\lambda},\rho_1)=\C\dvector{1\\1} \\
\Eig(\res_{z_2}\nabla^{u,\lambda},\rho_2)&=\C\dvector{0\\1}\,\,\,\,\Eig(\res_{z_3}\nabla^{u,\lambda},\rho_3)=\C\dvector{u\\1},\\
\end{split}
\end{equation}
and the parabolic structure with prescribed parabolic weights $(\rho_0,-\rho_0),$ .. $(\rho_3,-\rho_3)$ is determined
by the cross-ratio of these four lines considered as points in $\C P^1$, i.e., by 
\[Xratio([1:0],[1:1];[0:1],[u:1])=u.\]

\subsection{Stability}\label{stability}
Next, we determine which parabolic structures induced by $\nabla^u$ are stable. For $u\notin\{0,1,\infty\}$ every holomorphic line subbundle $L\subset\underline\C^2$ of degree $0$ meets at most one eigenline and we obtain
\[\pdeg(L)\leq -\rho_{\sigma(0)}-\rho_{\sigma(1)}-\rho_{\sigma(2)}+\rho_{\sigma(3)}\]
for all permutations $\sigma\in\mathfrak S(\{0,1,2,3\}).$ Moreover, equality holds for the trivial line subbundle $L= Eig(\res_{z_{\sigma(3)}}\nabla^{u},\rho_{\sigma(3)}).$
Similarly, for $u\neq m$ every line subbundle $L\subset\underline\C^2$ of degree $-1$ meets at most three eigenlines, and we obtain
\[\pdeg(L)\leq -1-\rho_{\sigma(3)}+\rho_{\sigma(0)}+\rho_{\sigma(1)}+\rho_{\sigma(2)}\]
for all permutations $\sigma\in\mathfrak S(\{0,1,2,3\})$ with equality for a suitable chosen bundle $L.$
For example, for $\sigma=\Id\in\mathfrak S(\{0,1,2,3\})$ $L$ is the tautological line bundle, i.e., its fiber at $[z:1]\in\CP^1$ is given by 
 \[L_{[z:1]}=\C\dvector{z\\ 1}.\]
From our assumption $\rho_i\in]0;\frac{1}{2}[$ we automatically have that $\pdeg L<0$ for all line subbundles $L\subset \underline\C^2$ of degree less or equal to $-2$.   Therefore, for $u\notin\{0,1,m,\infty\},$ stability of the parabolic structure induced by
$\nabla^{u,\lambda}$ is equivalent to the Biswas conditions \eqref{Biswas_criterion}.

For $u\in\{0,1,\infty\},$ there is a unique (trivial) line subbundle $L\subset\underline\C^2$ such that
$L$ meets two eigenlines, and we obtain
\[\pdeg(L)=-\rho_{\sigma(0)}-\rho_{\sigma(1)}+\rho_{\sigma(2)}+\rho_{\sigma(3)}\]
for a suitable $\sigma\in\mathfrak S(\{0,1,2,3\}).$ Thus 
we obtain (under the extra condition $1+\rho_{\sigma(3)} \geq \rho_{\sigma(0)}+\rho_{\sigma(1)}+\rho_{\sigma(2)}$) that:
\begin{itemize}
\item the parabolic structure induced by $\nabla^u$ for $u=0$ is (semi-)stable if and only if 
$\rho_2+\rho_3<(\leq)\,\, \rho_0+\rho_1;$
\item the parabolic structure induced by $\nabla^u$ for $u=1$ is (semi-)stable if and only if
$\rho_1+\rho_3<(\leq)\,\, \rho_0+\rho_2;$
\item the parabolic structure induced by $\nabla^u$ for $u=\infty$ is (semi-)stable if and only if
$\rho_0+\rho_3< (\leq)\,\, \rho_1+\rho_2.$
\end{itemize}

For $u=m$ the tautological line bundle $L$ meets all four eigenlines. Hence, we obtain that
the parabolic structure induced by $u=m$ is (semi-)stable if and only if
\[\rho_0+\rho_1+\rho_2+\rho_3<(\leq) 1.\]

\section{Abelianization of Fuchsian systems}\label{sec:abel}

Let $\nabla$ be a Fuchsian system as in \eqref{Fuchs2} such that the induced parabolic structure is (semi-)stable.
Assume there is a parabolic Higgs field  $\Psi$ with respect to the given parabolic structure such that
\[\det\Psi=\frac{(dz)^2}{z (z-1)(z-m)},\]
where $0,1,\infty,m\in\CP^1$ are the singularities of $\nabla.$ The eigenlines of $\Psi$ 
are well-defined on a double covering of $\CP^1$ branched at $0,1,\infty,m,$ i.e., on a complex torus
$T^2=\C/\Gamma$ of dimension $1.$ Without loss of generality we can assume
$\Gamma=\text{span}{(1,\tau)}$ and we can choose the elliptic involution $\sigma$ with respect to $\pi\colon\C/\Gamma\mapsto\CP^1$ to be 
$[w]\mapsto[-w].$ We can also fix our notations such that the preimage of $z_0$ is $w_0:=[0]\in\C/\Gamma,$ the preimage of $z_1$ is $w_1:=[1/2],$ the preimage of $z_2$ is $w_2:=[1/2+\tau/2],$ and the preimage of $z_3$ is $w_3:=[\tau/2].$ The eigenlines 
$L^\pm$ of $\pi^*\Psi$ have degree $-2$ as they intersect each other with order $1$ at $w_0,...,w_3,$
and because $\sigma ^*L^\pm=L^{\mp}.$
Note that \[L^+\otimes\sigma^*L^+=L^+\otimes L^-=L(-w_0-...-w_3).\] 
Let $S:=L(-2w_0)=...=L(-2w_3).$ Then we have $\sigma^*S=S$ and $\sigma^*S\otimes S=L(-w_0-...-w_3).$
The latter equation holds because there exists a meromorphic function (the derivative of the $\wp-$function) with a pole of order $3$ at $w_0$ and simple zeros at $w_1,w_2,w_3.$
Altogether, we see that for any parabolic Higgs field $\Psi$ with $\det \Psi\neq0$ the eigenlines
$L^\pm$ of $\Psi$ are given by \[L^+=S\otimes E,\,\,\,\,\ L^-=S\otimes E^*\] for a suitable $E\in \text{Jac}(\C/\Gamma).$
Moreover, $E$ is unique up to $E\mapsto E^*.$

Note that there is a unique meromorphic connection $\nabla^s$ on $S^*$ such that
the meromorphic connection $\nabla^s\otimes \nabla^s$ on $(S^*)^2=L(w_0+..+w_3)$ annihilates the holomorphic section 
$s_{w_0+..+w_3}$ with simple zeros at $w_0,..,w_3$. 
After pulling back a Fuchsian system $\nabla$ to $\C^2\to\C/\Gamma,$
$\nabla$ induces a unique meromorphic connection on the direct sum $L^+\oplus L^-$ of the eigenlines of $\Psi$
which makes the inclusion $L^+\oplus L^-\to\C^2$ parallel.
Tensoring the meromorphic connection on $L^+\oplus L^-$ with the flat line bundle $(S^*,\nabla^s)$ yields
a flat meromorphic connection $\hat\nabla$ on
\begin{equation}\label{hat_nabla}
E\oplus E^*\to \C/\Gamma.\end{equation}
The precise form of $\hat\nabla$ will be determined in the subsections \ref{explicit2}, \ref{Residues_of-Nabla} and \ref{2ff}.

\subsection{Concrete formulas II}\label{explicit2}
We first investigate the relationship between parabolic structures (in terms of the parameter $u$) and the holomorphic
eigenline bundles on the torus:
For a parabolic structure induced by $\nabla^u$ the Higgs field $\Psi^u$ in \eqref{psi^u} has determinant
\[\det\Psi^u=u(u-1)(m-u)\frac{(dz)^2}{(z(z-1)(z-m)}.\]
Its eigenlines are defined on the elliptic curve $T^2=\C/\Gamma$ given by the equation
\[y^2=z(z-1)(z-m).\]
The eigenvalues of $\Psi^u$ are 
\[\mp\sqrt{u(u-1)(u-m)}\frac{dz}{y},\]
i.e., constant multiples of the non-vanishing holomorphic differential $\frac{dz}{y}.$ 
The eigenline bundles $L^\pm$ of $\Psi^u$ are given by
\[\C \dvector{(-1+m) u z\mp \sqrt{u (u-1)(u-m)z(z-1)(z-m)}\\ -u z+m(-1+u+z)},\]
and their degree is $-2.$ Moreover, the divisors representing these line bundles are 
\[D^\pm=-3w_0+P^\pm, \]
where $w_0\in T^2$ is the point lying over $z=\infty,$
and $P^+=(z^+,y^+)$ and  $P^-=(z^+,y^-)$ are given with respect to the equation $y^2=z(z-1)(z-m)$ by
\begin{equation}
\label{yz-uv}
z^+=\frac{m-m u}{m-u},\,\,y^\pm=\pm\frac{m(m-1)}{(u-m)^2}v,\end{equation}
where \[v^2=u(u-1)(u-m)\]
is the algebraic equation for the Jacobian $\text{Jac}(T^2)$.

\subsection{Residues of $\hat\nabla$ on $\C/\Gamma$}\label{Residues_of-Nabla}
The following computation determines the residues of the connection $\hat\nabla$ at the points $w_i$  on the bundle $E \oplus E^*$ in \eqref{hat_nabla}:
There exists a local coordinate $w$ on $T^2\to\CP^1$ such that
$w^2=(z-z_i)$ together with a basis of $\C^2$ such that the pull-back of the Higgs field (as a 1-form) expands as
\[\pi^*\Psi=\dvector{o(w)& -2\frac{1}{w}+O(w)\\ -\frac{1}{2}w+o(w^2) &o(w)}dw.\]
Consider the (locally defined) gauge transformation
\[H=\dvector{1 & 1\\ -\frac{w}{2}& \frac{w}{2}}\]
with singularity at $w=0.$
Then,
\[H^{-1}\Psi H=\dvector{dw &0\\0&-dw}+\text{higher order terms}\]
and
\[H^{-1}dH=\dvector{\frac{1}{2} & -\frac{1}{2}\\-\frac{1}{2}&\frac{1}{2}}\frac{dw}{w}.\]

By definition the eigenline of the residue of $\nabla$ with respect to the positive eigenvalue $\rho_i>0$ lies in the kernel of residue of $\Psi,$
i.e., in the above mentioned frame, the pull-back of $\nabla$ is given by
\[\pi^*\nabla=d+\dvector{2\rho_i &0 \\ 0 &-2\rho_i}\frac{dw}{w}+\text{higher order terms}.\]
Applying the gauge transformation $H,$ we obtain
\[\pi^*\nabla.H=H^{-1}\circ\pi^*\nabla\circ H=d+\dvector{\frac{1}{2}& 2\rho_i-\frac{1}{2}\\2\rho_i-\frac{1}{2}&\frac{1}{2}}\frac{dw}{w}+\text{higher order terms}.\]

This computation together with the definition of $\nabla^s$ shows that the induced connection $\hat\nabla$ on $E\oplus E^*$ (as defined in \eqref{hat_nabla})
is given by
\[\hat\nabla=\dvector{\nabla^E &\beta^- \\ \beta^+ &\nabla^{E^*} },\]
where $\nabla^E$ is a smooth and holomorphic connection on $E,$ $\nabla^{E^*}$ is its dual on $E^*,$ 
and  $\beta^\pm$ are meromorphic $1-$forms with values in $E^{\mp2}$ such that $\beta^+ \otimes \beta^{-}$ has quadratic residues
given by
\begin{equation}\label{quadratic_residues}
\text{res}_{w_i}(\beta^+\beta^-)=(2\rho_i-\frac{1}{2})^2.\end{equation}

Altogether, we obtain via this abelianization procedure  a connection on $\C^2\to\C/\Gamma$ which
is (gauge equivalent to)
\begin{equation}\label{abel_connection}
\hat\nabla=\hat\nabla^{\alpha,\xi}=d+\dvector{\alpha dw-\xi d\bar w&\beta^-\\ \beta^+ &-\alpha dw+\xi d\bar w},\end{equation}
where $w$ is the global coordinate on $\C,$ $\alpha,\xi\in\C$ are suitable complex numbers, and
$\beta^\pm=\beta^\pm_\xi$ are meromorphic sections of the holomorphic line bundle given by the holomorphic structure
\[\dbar^\C\pm2\xi d\bar w\] with simple poles at $w_0,..,w_3.$    Using $\vartheta-$functions, we can write down
the second fundamental forms $\beta^\pm$ of $\hat\nabla$ with respect to the decomposition $E\oplus E^*$ explicitly as long as
$L(\dbar-\xi d\bar w)$ is not a spin bundle of $\C/\Gamma$:

\subsection{The second fundamental forms}\label{2ff}
Let $\vartheta$ denote the (shifted) $\vartheta-$function of $\C/ \Gamma$ for $\Gamma=\Z+\Z\tau.$ This means
that $\vartheta$ is 
the unique (up to a multiple constant) entire function satisfying $\vartheta(0) = 0$ and
\[\vartheta(w+ 1) = \vartheta (w),\,\,  \vartheta(w+ \tau) = - \vartheta (w)e^{-2\pi i w}.\]
Then the function \[t_{x}(w) = \frac{\vartheta(w-  x)}{\vartheta(w)}e^{\tfrac{2\pi i }{\bar\tau-\tau} x(w-\bar w)}\]
is doubly periodic on $\C\setminus\Gamma$  with respect to $\Gamma$
and satisfies 
\[(\dbar+\frac{2\pi i}{\bar \tau-\tau}xd\bar w)t_{x}=0.\] Thus $t_x$ is a meromorphic section of the bundle $\underline\C\to\C/\Gamma$ with respect to the holomorphic structure $\dbar+\frac{2\pi i}{\bar \tau-\tau}xd\bar w$
 and has a simple zero in $w=x$ and a first order pole  in $w = 0$ for $x  \notin\Gamma.$ 
\begin{Rem}
Note that the function $t_x$ gives an explicit realization of the two classical points of view on the moduli space of holomorphic line bundles: first, line bundles given by divisors, and second, line bundles given by $\dbar-$operators $\dbar+\frac{2\pi i}{\bar \tau-\tau}xd\bar w$ on
$\underline\C\to\C/\Gamma$ such that the Chern connection (with respect to the trivial metric ) is flat.
In other words, by fixing $[0]\in\C/\Gamma$ we have an identification of the torus $\C/\Gamma$ and
its Jacobian $\overline{H^0(\C/\Gamma, K)}/\Lambda$ with 
\[\Lambda=\{\bar\omega\in\overline{H^0(\C/\Gamma, K)}\mid \int_\gamma (\bar\omega-\omega)\in 2\pi i\Z \text{ for all closed curves } \gamma\}\]
via 
\[[x]\in\C/\Gamma\mapsto L([x]-[0])\cong L(\dbar+\frac{2\pi i}{\bar \tau-\tau}xd\bar w).\]
\end{Rem}

Using the functions $t_x$ we are able to write down the second fundamental forms $\beta^\pm$ explicitly:
\begin{Pro}\label{explicit_coeff}
 Let $x=\frac{\tau-\bar\tau}{2\pi i}\xi$ and assume that 
 $L(\dbar-\xi d\bar w)$ is not a spin bundle.  
 For $i=0,..,3$ set
\[\alpha^\pm_i=\alpha^\pm_i(x):=e^{\pm\frac{4\pi i}{\bar\tau-\tau}x(w_i-\bar w_i)}\frac{\vartheta(w_i\pm x)}{\vartheta(w_i\mp x)}\frac{\vartheta'(0)}{\vartheta(\pm2x)}(2\rho_i-\frac{1}{2}),\]
where $\vartheta'$ is the derivative of $\vartheta$ with respect to $w$
and $w_0=0,$ $w_1=\frac{1}{2},$ $w_2=\frac{1+\tau}{2}$ and $w_3=\frac{\tau}{2}.$ 
 Then the second fundamental forms $\beta^\pm_\xi$ in \eqref{abel_connection} are given by
 the 
 meromorphic $1-$forms
\[\beta^\pm_\xi([w])=\sum_{i=1}^4 \alpha^\pm_i(x)t_{\mp2x}(w-w_i)dw\]
 with values in the holomorphic bundle $L([\mp2x]-[0])=L(\dbar\pm2\xi)$ of degree $0.$
 \end{Pro}
  \begin{proof}
 The space $H$
 of meromorphic sections $\beta$ in $L^{\pm2}\otimes K$ with first order poles at $w_0,..,w_3$ is
 $4-$dimensional. If  $L=L(\dbar-\xi d\bar w)$ is not a spin bundle, the residue map
 \[H\to\C^4; \beta\mapsto (\res_{w_0}\beta,..,\res_{w_3}\beta)\]
 is an isomorphism. Therefore, the second fundamental forms $\beta^\pm$ are uniquely determined by their residues and 
 $\xi$.
 In order to determine the residues we are using the setup of section \ref{explicit2} and consider the meromorphic sections
\[s^\pm=\dvector{(-1+m) u z\mp \sqrt{u (u-1)(u-m)z(z-1)(z-m)}\\ -u z+m(-1+u+z)}\]
of the eigenline bundles $L^\pm$ of the parabolic Higgs field $\Psi^u$.
Recall that the divisors are given by
\[(s^\pm)=D^\pm=-3w_0+P^\pm, \]
where $w_0\in T^2$ is the point lying over $z=\infty$ and
$P^\pm=[\pm x]\in T^2=\C/\Gamma$ for a suitable $x\in\C.$ 
Computing the $\mathfrak{gl}(2 , \C)-$valued connection $1-$form of the Fuchsian system
$\nabla^u$ (or more generally $\nabla^{u,\lambda}$) with respect to the meromorphic frame $(s^+,s^-)$ 
and enables us to determine its residues at the preimages $[w_i]$ of the branch points $z_i:$
\begin{equation}\label{res_matrix}
\text{Res}_{[w_i]} \pi^* \nabla = \dvector{\frac{1}{2} & 2\rho_i- \frac{1}{2}\\ 2\rho_i- \frac{1}{2} &\frac{1}{2}}.\end{equation}
The connection $1-$form in \eqref{abel_connection} is then obtained by tensoring with the flat meromorphic line bundle connection $\nabla^S$ and by using the smooth frame 
\[(\tilde s^+,\tilde s^-)=( \frac{1}{t_x}s_{2w_0}\otimes s_{-3w_0+[x]},\frac{1}{t_{-x}}s_{2w_0}\otimes s_{-3w_0+[-x]}),\]
where the function $t_x$ is as in section \ref{2ff},
instead of the frame \[(s^+,s^-)=(s_{-3w_0+[x]},s_{-3w_0+[-x]}).\] 
 This implies that the lower left entry of the
residue matrix at $w_i$ of the connection 1-form with respect to $(\tilde s^+,\tilde s^-)$ is obtained from the lower left entry of the
residue matrix \eqref{res_matrix} at $w_i$ of the connection $1-$form with respect to $(s^+,s^-)$
by multiplying it with $\frac{t_{-x}(w_i)}{t_{x}(w_i)}.$ 
This observation together with a straight forward
computation imply the assertion.
\end{proof}

\begin{Rem} Note that this formula is in accordance
 with the formula for $\beta^\pm$ in \S{4} of \cite{He3} for the symmetric case where all local conjugacy classes are the same, i.e., $\rho_0=...=\rho_3.$  Furthermore,
in the case of $\rho_0=..=\rho_3=\frac{1}{4}$ one obtains abelian $\SLt$ connections (without singularities) on the 
torus $\C/\Gamma.$ This observation fits nicely with \S{6} of \cite{Gold}.
\end{Rem}

\subsection{Flat $\SLt$ connections on the $4-$punctured sphere in terms of flat line bundle connections on a torus}

We have seen in the previous section that flat line bundle connections on a torus uniquely determine (gauge equivalence classes of) flat $\SLt$ connections on the $4-$punctured 
sphere. This implies the following theorem:
\begin{The}\label{two-to-one}
Let $T^2\to\CP^1$ be the elliptic curve which is given by a double cover of the projective line
branched over $0,1,\infty,m\in\CP^1.$ Then \eqref{abel_connection}
gives rise to a $2-$to$-1$ correspondence between an open dense subset of the moduli space of flat
line bundles on $T^2$
and an open dense subset of the moduli space of  flat $\SLt$ connections on $\CP^1\setminus\{0,1,\infty,m\}$ whose local monodromies lie in the conjugacy classes prescribed by $\rho_i>0.$

This correspondence fails to exist exactly for the holomorphic spin bundles on $T^2$ respectively for
those flat $\SLt$ connections whose induced parabolic structure does not admit a
parabolic Higgs field with non-zero determinant.
\end{The}

Assume that the positive numbers $\rho_0,..,\rho_3$ satisfy the Biswas conditions.
Then a generic parabolic structure on the $4-$punctured sphere with parabolic weights
determined by $\rho_0,..,\rho_2$ and $\rho_3$ is stable.

We are going to extend theorem \ref{two-to-one} to the remaining stable parabolic structures. 
From section \ref{explicit2} we see that there are at most four stable parabolic structures which do not 
admit
a parabolic Higgs field with non-zero determinant. 
In terms
of the parameter $u$ of section \ref{explicit1} these parabolic structures are given by $u\in\{0,1,m,\infty\}.$ 
Further, \eqref{yz-uv} implies that the parabolic structures determined by $u\in\{0,1,m,\infty\}$
 correspond to the four spin bundles
on $T^2$ by choosing $S=L(-2w_0)$ as a base point in $Pic_{-2}(T^2).$

Before stating and proving the extension of the $2-$to$-1$ correspondence of theorem \ref{two-to-one} to the spin bundles  we first give a vague explanation of how to deal with the exceptional cases in theorem \ref{two-to-one}:
 If one takes a careful look at proposition 
\ref{explicit_coeff},
one sees that the second fundamental forms $\beta^\pm_\xi$ have a first order pole (in $\xi$) at the spin bundles 
$L(\dbar-\xi d\bar w).$ Hence, classical asymptotic analysis of ordinary differential equations (see for example \cite{Was})
indicate that the
complex linear part $\partial+\alpha dw$ of the line bundle connection
$d+\alpha dw-\xi d\bar w$
needs also to have a first order pole (in $\xi$) at the spin bundles $L(\dbar-\xi d\bar w):$
If $\alpha(\xi)$ is a (meromorphic) family such that the corresponding
$\SLt$ connections on the $4-$punctured $\CP^1$ extend through the spin bundles, then $\alpha$ must have first order poles and its residues can be computed to be
\begin{equation}\label{res_without_sign}
\begin{split}
&\pm(\hat\rho_0+\hat\rho_1+\hat\rho_2+\hat\rho_3) \frac{\pi i}{\tau-\bar\tau}  \hspace{2.25cm}\text{at } \xi=0\mod \Lambda (\Leftrightarrow u=m), \\
&\pm(\hat\rho_0+\hat\rho_1-\hat\rho_2-\hat\rho_3)\frac{\pi i}{\tau-\bar\tau}   \hspace{2.25cm}\text{at } \xi=\frac{\pi i\tau}{\tau-\bar \tau}\mod \Lambda (\Leftrightarrow u=\infty,) \\
&\pm(\hat\rho_0-\hat\rho_1+\hat\rho_2-\hat\rho_3) \frac{\pi i}{\tau-\bar\tau}  \hspace{2.25cm}\text{at } \xi=\frac{\pi i(1+\tau)}{\tau-\bar \tau}\mod \Lambda (\Leftrightarrow u=1), \\
&\pm(\hat\rho_0-\hat\rho_1-\hat\rho_2+\hat\rho_3)\frac{\pi i}{\tau-\bar\tau}   \hspace{2.25cm}\text{at } \xi=\frac{\pi i}{\tau-\bar \tau}\mod \Lambda (\Leftrightarrow u=0),\\
\end{split}
 \end{equation}
 where $\hat\rho_i=2\rho_i-\frac{1}{2}.$
 The following theorem rigorously proves \eqref{res_without_sign}. Moreover, it also determines
 which of the signs in \eqref{res_without_sign} does induce a stable parabolic structure and which does not.

\begin{The}\label{two-to-one_two}
The $2-$to$-1$ correspondence in theorem \ref{two-to-one} extends to the spin bundles $L(\dbar-\gamma d\bar w)$ (where $\gamma\in\frac{1}{2}\Lambda$)
and to the remaining flat $\SLt$ connections as follows:
Consider a meromorphic family of flat line bundle connections \[\nabla^\xi=d+\alpha(\xi)dw-\xi d\bar w\]
on an open neighborhood  of $\gamma\in\frac{1}{2}\Lambda$ and its
induced family of flat 
$\SLt$ connections on $\CP^1\setminus\{0,1,\infty,m\}.$ 
Then, the gauge orbits of the $\SLt$ connections on the 4-
punctured sphere converge for $\xi\rightarrow \gamma\in\frac{1}{2}\Lambda$ against the gauge orbit of a Fuchsian 
system
with (semi)-stable parabolic structure 
if and only if $\alpha$ expands around $\xi=\gamma$ as
\begin{equation}\label{a_spin_expansion}
\alpha(\xi)\sim_\gamma\frac{2\pi i}{\tau-\bar\tau}\frac{\mu_\gamma}{\xi-\gamma}+\bar\gamma+\,\text{ higher order terms in } \xi,\end{equation}
where
\[\mu_\gamma=\left\{\begin{array}{cl} |1-\rho_0-\rho_1-\rho_2-\rho_3| & \mbox{if }\gamma\in\Lambda\\ 
|\rho_0+\rho_1-\rho_2-\rho_3| & \mbox{if } \gamma\in\frac{\pi i}{\tau-\bar\tau}+\Lambda\\
|\rho_0+\rho_2-\rho_1-\rho_3| & \mbox{if } \gamma\in\frac{\pi i(1+\tau)}{\tau-\bar\tau}+\Lambda\\
|\rho_0-\rho_1-\rho_2+\rho_3| & \mbox{if } \gamma\in\frac{\pi i\tau}{\tau-\bar\tau}\Lambda\\
\end{array}\right..
\]
\end{The}
\begin{proof}
As in the proof of proposition \ref{explicit_coeff} we
consider the meromorphic sections
\[s^\pm=\dvector{(-1+m) u z\mp  v y\\ -u z+m(-1+u+z)}\]
of the eigenline bundles $L^\pm$ of the Higgs field $\Psi^u$ with respect to the parabolic structure induced by $\nabla^u.$
Recall that 
\[y^2=z(z-1)(z-m)\]
and
\[v^2=u(u-1)(u-m)\]
are the algebraic equations for the torus $T^2\to \CP^1$ and its Jacobian,  respectively.
The divisors of the sections $s^\pm$ are given by
\[D^\pm=-3w_0+P^\pm, \]
where $w_0\in T^2$ is the point lying over $z=\infty$ and
$P^\pm=[\pm x]\in T^2=\C/\Gamma$ for a suitable $x\in\C.$ 
The $(z,y)-$coordinates  of $P^\pm$ satisfy
\[z^+=\frac{m-m u}{m-u},\,\,y^\pm=\pm\frac{m(m-1)}{(u-m)^2}v.\]

With respect to the meromorphic frame $(s^+,s^-)$ we can compute the $\mathfrak{gl}(2 , \C)-$valued connection $1-$form of the Fuchsian system
$\nabla^u$ (or more generally $\nabla^{u,\lambda}$).
The upper left entry of the connection $1-$form has
the following asymptotic behavior around $u=0,1,m,$  as a straight forward computation shows:
\begin{equation}\label{res_with_sign}
\begin{split}
& \frac{m (-\rho_0-\rho_1+\rho_2+\rho_3)}{2v}\frac{dz}{y}+\text{higher orders in }v  \hspace{4.5cm} \text{at }  u=0,\\
& \frac{(m-1) (\rho_0-\rho_1+\rho_2-\rho_3)}{2v}\frac{dz}{y}+\text{higher orders in }v  \hspace{3.5cm}\text{at }  u=1,\\
& \frac{(m-1)m (-1+\rho_0+\rho_1+\rho_2+\rho_3)}{2v}\frac{dz}{y}+\text{higher orders in }v  \hspace{2.25cm}\text{at }  u=m.\\
\end{split}
 \end{equation}

We now identify the torus $T^2$ with its Jacobian via
\[x\in T^2\mapsto L(x-w_0)\in Jac(T^2),\] and expand
\eqref{res_with_sign} in terms of $x.$
To do so, we make use of the Weierstrass $\wp-$function $\wp\colon \C/\Gamma\to\CP^1$ of the torus
$T^2=\C/\Gamma.$ The $\wp-$function is the only doubly periodic meromorphic function on $\C$ (with respect to $\Gamma$) with double poles at the lattice points and holomorphic elsewhere and whose expansion
at $x=0$ is $\wp(x)\equiv \frac{1}{x^2}+...\, .$ The $\wp-$function satisfies the differential equation
\[(\wp')^2=4 \wp^3-g_2\wp-g_3,\]
where the two complex numbers $g_2,g_3\in\C$ are the lattice invariants of  $\Gamma.$
In terms of the $\wp-$function, the meromorphic functions $y,z\colon T^2\to\CP^1$ are given by
\begin{equation}\label{yz_P-rel}
\begin{split}
z&=\frac{\wp-p_2}{p_1-p_2}\\
y&=\frac{\wp'}{2(p_1-p_2)^\frac{3}{2}},\\
\end{split}
\end{equation}
where $p_i=\wp(w_i)$ and for a suitable choice of the square root of $p_1-p_2.$ Clearly, we have
\[m=z(w_3)=\frac{p_3-p_2}{p_1-p_2}.\]
Using $w$ as the affine coordinate of $\C,$ we obtain
\[\frac{dz}{y}=2\sqrt{p_1-p_2}dw\]
on $T^2=\C/\Gamma.$ Moreover, \eqref{yz-uv} yields that the complex parameters $(u,v)$ of
the space of line bundles can be expressed in terms of the $(y,z)-$parameters of the zero $p$ in the divisor
$D=p-3 w_0$ representing the line bundle $L_{(u,v)}:$
\[u=\frac{m(1-z)}{m-z}\,\,\,\,\text{ and }\,\,\,\, v=\frac{m (m-1)}{(z-m)^2}y.\]
Thus, we can expand \eqref{res_with_sign} in terms of $x$ as follows:
\begin{equation}\label{res_with_sign2}
\begin{split}
 \frac{m (-\rho_0-\rho_1+\rho_2+\rho_3)}{2v(x)}\frac{dz}{y}=&
(\rho_0+\rho_1-\rho_2-\rho_3)dw\frac{1}{x}+\text{higher orders in }x\\   &\text{at }  u=0, \text{or equivalently, } x=w_1;\\
 \frac{(m-1) (\rho_0-\rho_1+\rho_2-\rho_3)}{2v(x)}\frac{dz}{y}=&
(\rho_0+\rho_2-\rho_1-\rho_3)dw\frac{1}{x}+O(x)\\   &
\text{at }  u=1, \text{ or equivalently, } x=w_2;\\
 \frac{(m-1)m (-1+\rho_0+\rho_1+\rho_2+\rho_3)}{2v(x)}\frac{dz}{y}=&
(1-\rho_0-\rho_1-\rho_2-\rho_3)dw\frac{1}{x}+O(x)\\   &
\text{at }  u=m, \text{ or equivalently, }  x=w_0.
\end{split}
 \end{equation}
 We prefer to parametrize the Jacobian $Jac(T^2)$ in terms of $\xi$ via the $\dbar-$operator \[\dbar-\xi d\bar w.\]
 From section \ref{2ff} we obtain that
 \[L(x-w_0)=L(\dbar-\xi d\bar w)\]
 if and only if
 \[x=\frac{\tau-\bar\tau}{2\pi i}\xi\]
 up to adding lattices points of $\Gamma$ and $\Lambda,$ respectively .
 This  already yields the formula \eqref{res_without_sign}. 
 
 It remains to show for which choice of the sign in \eqref{res_without_sign} the corresponding parabolic structure is (semi)-stable:
 The parabolic structure for $u=0$ is (semi)-stable if and only if 
 $\rho_2+\rho_3\leq\rho_0+\rho_1.$ If this inequality holds, the first formula in \eqref{res_with_sign} determines the sign 
at $u=0.$ If this inequality is not satisfied for the parabolic weights $\rho_0,..,\rho_3,$ we 
 have used wrong coordinates $(u,\lambda)$ to parametrize the Fuchsian system (as the Fuchsian system 
 is not semi-stable at $u=0$). Using more appropriate coordinates
 $(\tilde u,\tilde\lambda)$ we obtain that the equation \eqref{a_spin_expansion} for $u= 0$ (or equivalently at $\gamma\in\frac{\pi i}{\tau-\bar\tau}+\Lambda$) also holds in the case of  $\rho_2+\rho_3\geq\rho_0+\rho_1.$
Similarly, one obtains the respective equations \eqref{a_spin_expansion} at $u=1$ and $u=\infty.$

 The parabolic structure for
 $u=m$ is (semi-)stable if and only if $\rho_0+\rho_1+\rho_2+\rho_3\leq1.$ If this inequality
 holds, the third formula in \eqref{res_with_sign} determines the sign 
at $u=m$ (or equivalently at $\gamma\in\Gamma$). 
If the inequality does not hold, we can argue as in the case of $u=0,1,\infty$  to obtain
the respective expansion \eqref{a_spin_expansion} at $u=m.$ 

 In order to determine the $0$th order term $\bar\gamma$ in \eqref{a_spin_expansion} we first note that
 the (non-zero) parabolic Higgs field $\Phi^u$ is diagonal with respect to the frame $(s^+,s^-)$ and
 its eigenvalues are $\pm v\frac{dz}{y}.$ Thus, adding a (non-zero) parabolic Higgs field to a Fuchsian system
 $\nabla^{u,\lambda}$ for $u\in\{0,1,m,\infty\}$ effects only the higher order terms in \eqref{a_spin_expansion}
 and not the constant order term. The constant order term can be computed similarly as the residue terms by
 using the frame $(\tilde s^+,\tilde s^-)$ in the proof of proposition \ref{explicit_coeff}.
\end{proof}

\begin{Rem}\label{para_moduli}
Note that theorem \ref{two-to-one} also induces a $2-$to$-1$ correspondence between the Jacobian of $\C/\Gamma$ and the moduli space 
$\mathcal M^{par}$ of (semi-)stable parabolic structures with prescribed parabolic weights (satisfying the Biswas conditions) on the $4-$punctured sphere.

Further, it is worth to mention that the $2-$to$-1$ correspondence from theorem \ref{two-to-one} extends
to flat $\SLt$ connections whose underlying parabolic structures are not semi-stable. In fact,
 the only difference to the case of theorem \ref{two-to-one_two} is that the residue terms in \eqref{a_spin_expansion}
 change their signs.
\end{Rem}

\begin{Rem}
At least for rational weights, there is another way to prove theorem \ref{two-to-one_two}: As in
\cite{Bis4} we can think of the moduli space of parabolic bundles as orbifold bundles, parabolic stability
reduces to the stability of a vector bundle on a suitable compact covering and we can adapt the proofs of \S 5
in \cite{He3} to this situation.
\end{Rem}

\subsection{Darboux coordinates}\label{sec:Darboux_coord}
We briefly recall the construction of the holomorphic symplectic structure on the moduli space of flat $\SLt$ connections on a 
punctured Riemann surface,
for details see \cite{At,AlMa} or alternatively \cite{Bis1, Bis2}. 

Via trace we identify $\mathfrak g:=\mathfrak{sl}(2,\C)$ and $\mathfrak g^*.$ Hence, the adjoint orbit of a diagonal 
$\mathfrak{sl}(2,\C)-$matrix with eigenvalues $\pm\rho_i$
inherits as a coadjoint orbit the Kirillov symplectic structure. We denote our adjoint orbits with respect to given eigenvalues $\pm\rho_i$
by 
$\mathcal O_0,..,\mathcal O_3,$ and consider
the space $\mathcal A_4$ which consists of connections $\nabla$ on a $4-$punctured Riemann surface $\Sigma$ of the form
\begin{equation}\label{residue_terms_Ai}
\nabla=A_i\frac{dz}{z-z_i}+\tilde\nabla^i, \end{equation}
where $\tilde\nabla^i$ extends smoothly to $z_i,$ $z$ is a local holomorphic coordinate around $z_i,$ 
and $A_i\in\mathcal O_i\subset\mathfrak{sl}(2,\C).$ On
\[\hat{\mathcal A}=\mathcal A_4\times\mathcal O_0\times..\times\mathcal O_3\]
we consider the symplectic form
\begin{equation}\label{symplectic_form_deco}
\Omega=\omega_\Sigma+\omega_0+..+\omega_3,
\end{equation}
where $\omega_i$ is the Kirillov form on $\mathcal O_i$ and
\begin{equation}\label{omega_integral}
\omega_{\Sigma} (A,B)=-\int_{\Sigma} tr(A\wedge B)
\end{equation}
for  tangent vectors $A,B$ on $\mathcal A_4$ considered as 
$A,B\in\Omega^1(\Sigma\setminus\{z_0,..,z_3\},\mathfrak{sl}(2,\C)).$ 
 The natural gauge action of $\mathcal G=\Gamma(\Sigma,\SLt)$ on $\hat{\mathcal A}$ has
 a moment map $\mu$ which is (in an appropriate sense) the sum of the curvature of $\nabla,$  of  the residues of $\nabla$ and of the moment maps of the coadjoint orbits.
Then, the symplectic space $\mu^{-1}\{0\}$ is the moduli space $\mathcal{A}=\mathcal{A}^{\rho_0,..,\rho_3}_{\Sigma\setminus\{z_0,..,z_3\}}$ of flat connections on the $4-$punctured Riemann surface whose local monodromies are 
 determined by the $\pm\rho_i.$

 We are mainly interested in the case of the 4-punctured sphere, and we pull back connections to the torus by the double covering $T^2\to\CP^1$ 
 branched over the singular points $z_0,..,z_3.$
 \begin{Lem}
The symplectic structure $\tilde \Omega$
 of the moduli space 
 of flat $\SLt$ connections on the 4-punctured torus $\mathcal{A}^{2\rho_0,..,2\rho_3}_{T^2\setminus\{w_0,..,w_3\}}$
 restricted to the subspace of connections obtained by pull-back $\pi^*$,
 is twice the symplectic structure $\Omega$ of the moduli space 
 $\mathcal{A}^{\rho_0,..,\rho_3}_{\CP^1\setminus\{z_0,..,z_3\}}$ of flat $\SLt$ connections 
 on the 
 $4-$punctured sphere, i.e.,
 \[ \tilde \Omega_{[\pi^*\nabla]}(\pi^*  X,\pi^*  Y)=2\Omega_{[\nabla]}( X,  Y)\]
 for all $X,Y\in T_{[\nabla]}\mathcal{A}^{\rho_0,..,\rho_3}_{\CP^1\setminus\{z_0,..,z_3\}}.$
 \end{Lem}
 \begin{proof}
 Consider $\nabla$ in the gauge orbit $[\nabla]\in\mathcal{A}^{\rho_0,..,\rho_3}_{\CP^1\setminus\{z_0,..,z_3\}},$ and let
 $A_0,..,A_3$ be its residue terms at $z_0,..z_3.$
 It is well-known (and easy to prove) that in every gauge equivalence class $[\tilde\nabla]\in\mathcal{A}^{\rho_0,..,\rho_3}_{\CP^1\setminus\{z_0,..,z_3\}}$ there is a representant $\tilde\nabla$ whose residue terms are also $A_0,..,A_3.$
 Hence, we can restrict to connections with fixed residue terms $A_0,..,A_3.$ This implies that
 tangent vectors
 $X,Y\in T_{[\nabla]}\mathcal{A}^{\rho_0,..,\rho_3}_{\CP^1\setminus\{z_0,..,z_3\}}$  can be
 represented
 by smooth sections
 $A,B\in\Omega^1(\CP^1,\mathfrak{sl}(2,\C)),$ and for computing the symplectic form  we do not need to take the boundary terms $\omega_0+..+\omega_3$ into account, i.e., we get
 \[\Omega_{[\nabla]}( X,  Y)=-\int_{\CP^1}tr(A\wedge B).\]
On the other hand, the tangent vectors
$\pi^*X,\pi^*Y\in T_{[\pi^*\nabla]}\mathcal{A}^{2\rho_0,..,2\rho_3}_{T^2\setminus\{w_0,..,w_3\}}$
are represented by the smooth pull-backs $\pi^*A,\pi^*B\in\Omega^1(T^2,\mathfrak{sl}(2,\C)),$
and integration yields  
\[\tilde\Omega_{[\pi^*\nabla]}( \pi^*X,  \pi^*Y)=-\int_{T^2}tr(\pi^*A\wedge \pi^*B)=-2\int_{\CP^1}tr(A\wedge B)\]
as $\pi\colon T^2\to\CP^1$ is a double covering. 
  \end{proof} 
  
  Note that
tensoring with the flat line bundle $(S^*,\nabla^S)$ provides
a symplectomorphism between the corresponding moduli spaces of flat $\SLt$ connections with 
  prescribed conjugacy classes of the local monodromies.
Hence our discussion together with theorem \ref{two-to-one} show that the moduli space of flat line bundle connections on $T^2$ 
 provides a concrete realization of the space $\mathcal{A}^{\rho_0,..,\rho_3}_{\CP^1\setminus\{z_0,..,z_3\}}$ (as a double covering) and the symplectic form $\Omega$
 can be easily computed in terms of the coordinates $(\alpha,\xi)$ on the moduli space of flat line bundle connections (provided by
theorem \ref{two-to-one} and \eqref{abel_connection}).
In fact, the Kirillov residual terms $\omega_i(\frac{\partial }{\partial \alpha},\frac{\partial }{\partial \xi})$ vanish
and the surface term computes as
\begin{equation*}
\begin{split}
\omega_{\C/\Gamma}(\frac{\partial }{\partial \alpha},\frac{\partial }{\partial \xi})
=-\int_{\C/\Gamma}tr\left(\frac{\partial \hat\nabla^{\alpha,\xi}}{\partial \alpha}\wedge\frac{\partial \hat\nabla^{\alpha,\xi}}{\partial \xi}\right)
=2\int_{\C/\Gamma}dw\wedge d\bar w,
\end{split}
\end{equation*}
 where $\hat\nabla^{\alpha,\xi}$ are the connections given by \eqref{abel_connection}.
 Thus we obtain the following theorem:
 \begin{The}\label{Darboux_coord}
 In terms of the coordinates $\alpha,\xi$ (provided by
Theorem \ref{two-to-one} and \eqref{abel_connection}) the holomorphic symplectic form $\Omega$ on the moduli space 
$\mathcal A$ of flat $\SLt$ connections on the 4-punctured
 sphere with prescribed local monodromies (determined by $\pm\rho_i\in\R$) is given  by
 \[\Omega=\left(\int_{\C/\Gamma}d w\wedge d \bar w\right)d\alpha \wedge d\xi.\]
 \end{The}

 \section{On Witten's formula for the symplectic volume of the moduli space
of flat connections}\label{symp_volume}

The moduli space of unitary connections on the 4-punctured sphere with prescribed local monodromy conjugacy classes $\tilde{\mathcal M}$ can be naturally considered (away from its singularities) as
a symplectic manifold, see for example \cite{ZT,Bis1,Bis2}.
As before, we identify it with
the moduli space $\hat {\mathcal M}$ of unitarizable Fuchsian systems with prescribed local monodromy conjugacy classes
(determined by $\rho_i\in]0;\frac{1}{2}[$). Moreover, by \cite{MS,Biq}, we can identify $\hat {\mathcal M}$ with the moduli space of parabolic structures with prescribed parabolic weights. The latter space is a complex analytic space and the symplectic structure is 
a K\"ahler form \cite{Bis2}.

The K\"ahler structure  on the moduli space $\mathcal M$ of parabolic bundles is the 
restriction of the holomorphic 
symplectic form on the moduli space $\mathcal A$ of flat $\SLt$ connections on the 4-punctured sphere to the (real analytic)
sub-variety consisting of flat connections with unitarizable monodromy \cite{ZT,Bis1,Bis2}.
Because of the Riemann-Hilbert correspondence we can identify $\mathcal A$ with the space of Fuchsian 
systems considered in \S\ref{sec:fuchs}.
The map from the moduli space $\mathcal A$ of flat $\SLt$ connections to the moduli space $\mathcal M$ of parabolic 
structures is a holomorphic 
fibration \cite{ZT, Bis2, ArLy}. A fiber over a stable parabolic structure is an affine space whose underlying vector space is
the space of parabolic Higgs fields, or, by Serre duality, the cotangent space to the moduli space of parabolic structures.

By the Mehta-Seshadri Theorem \cite{MS} (or \cite{Biq} in the general case of irrational weights $\rho_i$)
there is a unique compatible flat irreducible connection with unitarizable monodromy representation for every stable parabolic structure on the $4-$punctured $\CP^1$ with 
 parabolic weights $\rho_i$.
This correspondence can be interpreted as a section \[\varphi_{MS}\colon\mathcal M\to\mathcal A\] of the affine bundle 
$\mathcal A\to\mathcal M,$ see for example \cite{ZT,Bis2} or for the particular case of compact surfaces without punctures Chapter 3 in \cite{Ty}.
 As the elements of the Jacobian of $T^2=\C/\Gamma$ parametrize the moduli space of
 parabolic structures on the 4-punctured sphere (see Remark \ref{para_moduli}), there exists for every $\xi\in\C\setminus\frac{1}{2 d\bar w}\Lambda$ 
a unique
 $\alpha^{MS}(\xi)\in\C$ such that the connection $\hat\nabla^{\alpha^{MS}(\xi),\xi}$ in \eqref{abel_connection} corresponds to a unitarizable connection on the $4-$punctured sphere.
We consider  the moduli space $\mathcal A^1_{\C/\Gamma}$ of flat line bundles over $\C/\Gamma$
as a   holomorphic fibration over the Jacobian and
 obtain a real analytic
 section \[\alpha^{MS}\colon Jac(\C/\Gamma)\to \mathcal A^1_{\C/\Gamma}.\]
 By definition, $\alpha^{MS}$ is a lift of the section $\varphi_{MS}$
 to the double covering $Jac(\C /\Gamma) \rightarrow \C P^1=\mathcal M.$
 Since $\mathcal A\to\mathcal M$ is a holomorphic affine bundle, \[\dbar  \varphi_{MS}\] is a well defined section in  
 $\Omega^{(0,1)}(\mathcal M, T^{(1,0)}\mathcal M ^*) \cong \Omega^{(1,1)}(\mathcal M, \C).$ 
 In fact, it is the K\"ahler form up to a constant multiple, see \cite{ZT,Bis2}.
In our setup we obtain this property of $\dbar  \varphi_{MS}$ as a corollary of
 theorem \ref{Darboux_coord}:

\begin{Cor}\label{Kaehler_form}
In terms of the coordinate $\xi$ for the Jacobian of $\C/\Gamma$
(and hence for $\mathcal M$, see Remark \ref{para_moduli}) the K\"ahler form on the moduli space of
parabolic structures $\mathcal M$ with given parabolic weights is
\[\omega=(\int_{\C/\Gamma}d w\wedge d \bar w) \dbar \alpha^{MS}\wedge d\xi,\]
where \[\dbar \alpha^{MS}=\frac{\partial \alpha^{MS}(\xi)}{\partial \bar \xi}d\bar \xi\in\Omega^{(0,1)}(Jac(\C/\Gamma))\] is the natural derivative in the affine holomorphic bundle $\mathcal A^1_{\C/\Gamma}\to Jac(\C/\Gamma).$
\end{Cor}


We want to compute the symplectic volume $\int_{\mathcal M} \omega$ of the moduli space in terms of
the free parameters $\rho_i.$ 
The formula (in its general form for $n-$punctured surfaces of genus $g$) is 
known as Witten's formula stated and rigorously proven in \cite{W}.  Alternative proofs were for example given in \cite{JW, TZ}. 
Our proof uses the herein developed abelianization method and seems to shed new light on the K\"ahler geometry of $\mathcal M.$

\begin{The}[Witten's formula]
Let $\mathcal M$ be the moduli space of parabolic structures on $\CP^1\setminus\{z_0,..,z_3\}$  with
parabolic weights $\rho_i\in]0;\frac{1}{2}[,\, i=0,..,3$ satisfying the Biswas conditions \eqref{Biswas_criterion} for stability and let $\omega$
be its natural K\"ahler form. 
Then its symplectic volume is given by
\begin{equation}\label{Witten_vs_Heller}
\vol(\mathcal M)=2\pi^2(1-\mu_0-\mu_1-\mu_2-\mu_3),\end{equation}
where \[\mu_0=|1-\rho_0-\rho_1-\rho_2-\rho_3|,\,\,\, \mu_1=|\rho_0+\rho_1-\rho_2-\rho_3|,\,\,\, \mu_2=
|\rho_0-\rho_1+\rho_2-\rho_3|\]
and
\[ \mu_3=|\rho_0-\rho_1-\rho_2+\rho_3|.\]


\end{The}
\begin{proof}
We make use of the global coordinate $\xi$ on the universal covering $\bar{H^0(\C/\Gamma,K)}$ of the Jacobian of $\C/\Gamma$ via the parametrization of holomorphic structures $\dbar^\xi=\dbar-\xi d\bar w.$ Recall, that we assume without loss of generality, that  $\Gamma=\Z+\Z\tau.$
Thus the flat line bundle connections 
\[\nabla=d+\alpha d w-\xi d\bar w,\]
\[\nabla .g_1=d+ (\alpha+\frac{2\pi i}{\tau-\bar\tau}\bar\tau) d w-(\xi+\frac{2\pi i}{\tau-\bar\tau}\tau) d\bar w\]
and
\[\nabla .g_2=d+(\alpha+\frac{2\pi i}{\tau-\bar\tau}) d w-(\xi+\frac{2\pi i}{\tau-\bar\tau}) d\bar w \]
are gauge equivalent on $\C/\Gamma$, where \[g_1=\exp(\frac{2\pi i}{\tau-\bar \tau}(\bar\tau w-\tau\bar w))\] and
 \[g_2=\exp(\frac{2\pi i}{\tau-\bar\tau}(w-\bar w)).\]
Therefore, the section $\alpha^{MS}$, considered as a function in terms of $\xi$, satisfies
the functional equations
\begin{equation}\label{functional_equation_a}
\alpha^{MS}(\xi+\frac{2\pi i}{\tau-\bar\tau}\tau)=\alpha^{MS}(\xi)+\frac{2\pi i}{\tau-\bar\tau}\bar\tau\end{equation}
and
\begin{equation}\label{functional_equation_b}
\alpha^{MS}(\xi+\frac{2\pi i}{\tau-\bar\tau})=\alpha^{MS}(\xi)+\frac{2\pi i}{\tau-\bar\tau}\end{equation}
for all $\xi\in\C\setminus\frac{1}{2 d\bar w}\Lambda.$ Note also, that $\alpha^{MS}$ is an odd function by construction.

Identifying the torus $\C/\Gamma$ and its Jacobian $\C/\frac{1}{d\bar w}\Lambda$ once again via
\[[x]\in \C/\Gamma\mapsto \xi=\frac{2\pi i}{\tau-\bar\tau}x\]
we then may use the $\vartheta-$function of $\C/ \Gamma$ as a $\vartheta-$function on  $\C/\frac{1}{d\bar w}\Lambda$
as follows:
\[\theta(\xi):=\vartheta(\frac{\tau-\bar\tau}{2\pi i} \xi).\]
Clearly, $\theta(\xi+\frac{2\pi i}{\tau-\bar\tau})=\theta(\xi)$ and $\theta(\xi+\frac{2\pi i}{\tau-\bar\tau}\tau)=-\theta(\xi)\exp((\bar\tau-\tau)\xi).$

Then, we obtain from \eqref{a_spin_expansion} in Theorem \ref{two-to-one_two} 
and from the functional equations \eqref{functional_equation_a}, \eqref{functional_equation_b}
 that the section $\alpha^{MS}$ considered as a function
on the universal covering of the Jacobian of $\C/\Gamma$ 
can be written
as
\begin{equation}\label{mehta_seshadri_ex}
\alpha^{MS}(\xi)=(\sum_{i=0}^3\mu_{\gamma_i})\xi+(1-\sum_{i=0}^3\mu_{\gamma_i})\bar\xi +f(\xi)+\sum_{i=0}^3\frac{2\pi i}{\tau-\bar\tau}\frac
{\mu_{\gamma_i}}{2}(\frac{\theta'(\xi-\gamma_i)}{\theta(\xi-\gamma_i)}-\frac{\theta'(-\xi-\gamma_i)}{\theta(-\xi-
\gamma_i)}),
\end{equation}
where $f$ is a doubly periodic (with respect to $\frac{1}{d\bar w}\Lambda$) function,
\[\gamma_0=0,\,\,\, \gamma_1=\frac{\pi i}{\tau-\bar\tau},\,\,\, \gamma_2=\frac{\pi i}{\tau-\bar\tau}(1+\tau),\,\,\, \gamma_3=\frac{\tau \pi i}{\tau-\bar\tau},\]
and the $\mu_{\gamma_i}$ are as in Theorem \ref{two-to-one_two}.
From Corollary \ref{Kaehler_form} we obtain that
\begin{equation}\label{Kaehler_form_2}
\begin{split}
\omega&=(\int_{\C/\Gamma}d w\wedge d \bar w) \dbar \alpha^{MS}\wedge d\xi\\
&=(\int_{\C/\Gamma}d w\wedge d \bar w)((1-\sum_{i=0}^3\mu_{\gamma_i})d\bar\xi+d f)\wedge d\xi
\end{split}
\end{equation}
and integration yields
\begin{equation}\label{volume}
\begin{split}
\vol(\mathcal M)&=\frac{1}{2}\int_{\text{Jac}(\C/\Gamma)}\omega\\
&=\frac{1}{2}\left(1-\sum_{i=0}^3\mu_{\gamma_i}\right)\left(\int_{\C/\Gamma}d w\wedge d \bar w\right)\left(\int_{\C/\frac{1}{d\bar w}\Lambda}d\bar\xi\wedge d\xi\right)\\
&=2\pi^2\left(1-\sum_{i=0}^3\mu_{\gamma_i}\right)
\end{split}
\end{equation}
as claimed.
\end{proof}

\begin{Rem}
It is worth noting that our formula \eqref{Witten_vs_Heller} for the symplectic volume coincides
with Witten's formula up to a normalization constant of $4 \pi^2$. Since we restrict to the weights $\rho_0,..,\rho_3$ satisfying the Biswas conditions \eqref{Biswas_criterion},
the moduli space $\mathcal M$ is non-empty. We can assume without loss of generality that $\rho_0\leq \rho_1 \leq \rho_2 \leq \rho_3.$ 
Witten's formula gives:
\begin{equation}\label{wit}
Vol(\mathcal M) =8\sum_{n= 1}^{\infty}\tfrac{1}{n^2}\Pi_{i= 0}^3\sin(2\pi n \rho_i).\end{equation}

The dilogarithm $Li_2$ is defined to be the analytic continuation to $\C \setminus [1, \infty[$ of the function
\[Li_2(z)= \sum_{n= 1}^\infty \tfrac{z^n}{n^2} \quad \text{ for } |z| \leq 1,\]
continuous on the whole unit circle and 
satisfying the functional equation
\begin{equation}\label{dilog}
Li_2(z) + Li_2(\tfrac{1}{z}) = - \tfrac{\pi^2}{6} - \tfrac{1}{2} \log^2(-z),
\end{equation}
see \cite[Page 8 \& 11 ]{Z}. 
For $z=r e^{i\theta}$ with $2k \pi<\theta \leq 2(k+1) \pi$ we need to choose  $\log(-1) = -i\pi(2k+1)$
in \eqref{dilog}.
Since $\sin(2 \pi n \rho_i) = \tfrac{1}{2}i (e^{-i 2\pi n \rho_i} - e^{i 2\pi n \rho_i})$ we obtain from \eqref{wit} 
\begin{equation*}
\begin{split}
 Vol(\mathcal M) &= Li_2(e^{i2\pi (\rho_0 + \rho_1 + \rho_2 + \rho_3)}) + Li_2(e^{i 2\pi (-\rho_0 - \rho_1 - \rho_2 - \rho_3)}) - Li_2(e^{i2\pi (-\rho_0 + \rho_1 + \rho_2 + \rho_3)}) \\
&-Li_2(e^{i2\pi (\rho_0 - \rho_1 - \rho_2 - \rho_3)}) -Li_2(e^{i2\pi (\rho_0 - \rho_1 + \rho_2 + \rho_3)}) -Li_2(e^{i2\pi (-\rho_0 + \rho_1 - \rho_2 - \rho_3)}) \\
&-Li_2(e^{i2\pi (\rho_0 + \rho_1 - \rho_2 + \rho_3)}) -Li_2(e^{i2\pi (-\rho_0 - \rho_1 + \rho_2 - \rho_3)}) -Li_2(e^{i2\pi (\rho_0 + \rho_1 + \rho_2 - \rho_3)})  \\
&-Li_2(e^{i2\pi (-\rho_0 - \rho_1 - \rho_2 + \rho_3)}) +Li_2(e^{i2\pi (-\rho_0 - \rho_1 + \rho_2 + \rho_3)}) +Li_2(e^{i2\pi (\rho_0 + \rho_1 - \rho_2 - \rho_3)})\\
&+Li_2(e^{i2\pi (-\rho_0 + \rho_1 - \rho_2 + \rho_3)}) +Li_2(e^{i2\pi (\rho_0 - \rho_1 + \rho_2 -\rho_3)}) +Li_2(e^{i2\pi (-\rho_0 + \rho_1 + \rho_2 - \rho_3)}) \\
&+ Li_2(e^{i2\pi (\rho_0 - \rho_1 - \rho_2 + \rho_3)}).\\
\end{split}
\end{equation*}

Then there are four cases to consider:\begin{enumerate}
\item  $\;(\rho_0 + \rho_1 + \rho_2 + \rho_3) \leq  1\;$ and  $\;(-\rho_0 + \rho_1 +  \rho_2 - \rho_3) \leq 0$
\item  $\;(\rho_0 + \rho_1 + \rho_2 + \rho_3) \leq  1\;$ and  $\;(-\rho_0 + \rho_1 +  \rho_2 - \rho_3) > 0$
\item  $\;(\rho_0 + \rho_1 + \rho_2 + \rho_3) >  1\;$ and  $\;(-\rho_0 + \rho_1 +  \rho_2 - \rho_3)  \leq 0$
\item  $\;(\rho_0 + \rho_1 + \rho_2 + \rho_3) > 1\;$ and  $\;(-\rho_0 + \rho_1 +  \rho_2 - \rho_3) > 0$
\end{enumerate}

In the first case we obtain from \eqref{dilog} and by the Biswas conditions \eqref{Biswas_criterion}:
\begin{equation*}
\begin{split}
\tfrac{2}{\pi^2}Vol(\mathcal M) &=    (-1+ 2 (\rho_0 + \rho_1 + \rho_2 + \rho_3))^2 +  (-1+2 (-\rho_0 - \rho_1 + \rho_2 + \rho_3))^2\\
& + (-1 +  2 (-\rho_0 + \rho_1 - \rho_2 + \rho_3))^2 +   (-1- 2 (-\rho_0 + \rho_1 +  \rho_2 - \rho_3))^2\\
&-  (-1+2 (-\rho_0 + \rho_1 + \rho_2 + \rho_3))^2 - (-1+2 (\rho_0 - \rho_1 + \rho_2 + \rho_3))^2\\
&-  (-1+2 (\rho_0 + \rho_1 - \rho_2 + \rho_3))^2 -  (-1+2 (\rho_0 + \rho_1 + \rho_2 - \rho_3))^2\\
&= 8 (\rho_0 + \rho_1 + \rho_2 -\rho_3),
\end{split}
\end{equation*}
which coincides with \eqref{Witten_vs_Heller}. The other cases can be computed analogously.
\end{Rem}

\end{document}